\documentclass[12pt]{amsart}
\usepackage{templatestyle,templatecommand,templateomega,templatetensor}

\title[MY for compact K\"ahler manifolds with semi-positive $K_X$]
{Miyaoka-Yau inequality for compact K\"ahler manifolds with semi-positive canonical bundle} 

\author
{Ryosuke Nomura}

\address{Graduate School of Mathematical Sciences, The University of Tokyo \endgraf
	3-8-1 Komaba, Meguro-ku, Tokyo, 153-8914, Japan.}

\email{nomu@ms.u-tokyo.ac.jp}

\thanks{Classification AMS 2010: 53C55, 
	32W20.
}


\begin{document}
	
\begin{abstract}
	In this paper, we prove the Miyaoka-Yau inequality for compact K\"ahler manifolds with semi-positive canonical bundle.
	The key point of the proof is the estimate for the $L^2$-norm of the scalar curvature along the K\"ahler-Ricci flow.
\end{abstract}
	
\maketitle

\section{Introduction}
Let $X$ be a complex manifold of dimension $n$ and $K_X$ be the canonical bundle of $X$.
The Miyaoka-Yau inequality is the following inequality for Chern classes of $X$ which holds under suitable positivity condition on $K_X$: 
\begin{align}\label{MY}
\tag{MY}
	\left( 
		2(n+1)c_2 (X) - n c_1(X)^2
 \right)
	 \cdot \left( -c_1 (X) \right)^{n-2} 
 \ge 0.
\end{align}
In 1977, Yau \cite{Yau77} showed (\ref{MY}) under the ampleness of $K_X$ and Miyaoka \cite{Miyaoka77} under the bigness of $K_X$ and $n=2$. 
After that, Y. Zhang \cite{Zhang09MiyaokaYau} obtained (\ref{MY}) for smooth minimal projective varieties of general type and Guenancia-Taji \cite{GuenanciaTaji16} for minimal projective varieties.
For more detailed historical account and related results, we refer to \cite[Section 1]{GuenanciaTaji16}.
We remark that all most all results for (\ref{MY}) are proved under the projectivity assumption.
Our motivation here is to extend it to compact K\"ahler manifolds.
The main theorem is stated as follows: 
\begin{theorem}\label{Main}
	All compact K\"ahler manifolds with semi-positive canonical bundle
	satisfy (\ref{MY}).
\end{theorem}
\noindent
Here, semi-positive means that there exists a smooth Hermitian metric on $K_X$ whose Chern curvature is semi-positive.
It is natural to expect that (\ref{MY}) holds even when compact K\"ahler manifolds with nef canonical bundle.
However, in our argument, we only prove when $K_X$ is semi-positive.

Before we outline the proof of Theorem \ref{Main}, we fix some notations.
Let $X$ be a compact \kahler manifold of dimension $n$. 
For a \kahler form $\omega$ on $X$, we denote $\Rm(\omega )$ the Riemann curvature tensor of $\omega$, $\ric(\omega)\in 2\pi c_1(X)$ the Ricci curvature of $\omega$, and $R(\omega)$ the scalar curvature of $\omega$.
The \krf \ starting from a \kahler form $\omz$ is the smooth family $\{\omt\}_{t\ge 0} $ of \kahler forms satisfying
\begin{align*}
\begin{cases}
	\, \dt \omt 
		&=-\rict - \omt, \\[5pt]
	\, \omt|_{t=0}
		&=\omz ,
\end{cases}
\end{align*}
which we simply write $\omt$.

For the proof of Theorem \ref{Main}, we first note that since the case when $K_X$ is nef and big (this condition is equivalent to $K_X$ is semi-positive and big) is shown by \cite{Zhang09MiyaokaYau}, we only need to show (\ref{MY}) if $K_X$ is semi-positive and not big.
We follow the argument in \cite{Zhang09MiyaokaYau}.
The essential point of his proof is to reduce (\ref{MY}) to the uniform boundedness of the scalar curvature along the K\"ahler-Ricci flow 
which is shown by \cite{Zhang09ScalKRFMinGenType} when $K_X$ is nef and big. 
In our setting, we need the following new scalar curvature estimate.
\begin{theorem}[$=$Theorem \ref{Scalar}]
Let $\omz$ be a \kahler form satisfying $[\omz] - \ckx >0$ and $\omt$ be the \krf \ starting from $\omz$.
Assume that $K_X$ is semi-positive and not big.
Then the following estimate holds:
	\begin{align*}
		\int_0^\infty dt \intx \rt^2 \omtn
			&< \infty .
	\end{align*}
\end{theorem}
For the proof, we consider the function $\et$ defined by  
\begin{align*}
	\et 
		\deq \intx \ii \partial \ft \wedge \dbar \ft \wedge \omt^{n-1},
\end{align*}
where $f_t \deq \log (\omtn/\Ome)$ and $\Omega$ is a volume form on $X$ such that $-\ricOm \ge 0$.
The function $\et$ is the Dirichlet norm of $f_t$ and is also similar to the 1st derivative of the Mabuchi's energy.
The key observation is that the time derivative of $\et$ can be used to estimate the $L^2$-norm of the scalar curvature of $\omt$.

We remark that Song-Tian \cite{SongTian11ScalKRF} showed that if $K_X$ is semi-ample, then the scalar curvature is uniformly bounded along the \krf. 
However, we cannot apply this result to our setting. \\

\noindent\textbf{Acknowledgment} The author would like to thank his supervisor Prof. Shigeharu Takayama for various comments.
This work is supported by the Program for Leading Graduate Schools, MEXT, Japan.

\section{Proof of the Theorems}
We first recall the argument of \cite{Zhang09MiyaokaYau}. 
The following formula for Chern classes is well-known. 
\begin{proposition}[see {\cite[Chapter 4]{KobayashiBook}}]\label{Formula}
For any \kahler form $\omega $, the following estimate holds.  
	\begin{align*}
		&\left( 
		2(n+1)c_2 (X) - n c_1(X)^2
		 \right)
		 \cdot [\omega ]^{n-2} \\
		&=\dfrac{1}{4\pi^2 n(n-1) }
					\int_X \left(
							(n+1) \absoms{\Rmc(\omega )}^2 
							- (n+2) \absoms{\ricc(\omega )} ^2
					\right)	\omega^{n}\\
		&\ge	 \dfrac{1}{4\pi^2 n(n-1) }
							\int_X \left(
									(n+1) \absoms{\Rmc(\omega )}^2 
									- (n+2)  \absom{\ric (\omega)+\omega }^2
							\right)	\omega^{n}.
	\end{align*}
Here we set 
	\begin{align*}
		\omega	
			&=\ii \gijb \dzidzjb, \\
		\Rmc (\omega)_{i\jbar k \lbar}
			&\deq \Rm(\omega )_{i \jbar k \lbar} 
				- \dfrac{R(\omega)}{n(n+1)} (\gijb \gklb + g_{i \lbar} g_{k\jbar} ) ,\\
		\ricc(\omega)
			&\deq \ric(\omega) - \dfrac{R(\omega)}{n}\omega .
	\end{align*}
\end{proposition}

Thanks to this proposition, in order to prove (\ref{MY}), we only need to find \kahler forms $\{ \omega_i  \}_{i=1}^\infty$ satisfying  
\begin{align}
\label{omi} &\lim_{i \to \infty} [\omega_i] = -2\pi c_1(X)=2\pi c_1(K_X),\\
\label{Ricci}	&\lim_{i \to \infty} \intx \absomi{\ric (\omega_i )+\omega_i }^2 \omega_i^n
	= 0.
\end{align}
In the following argument, we will prove that the \krf \ $\omt$ satisfies these two conditions. 

We first recall the long time existence theorem for the \krf.
\begin{theorem}[{\cite{TianZhang06KRFProjGenType}}]\label{Existence}
For any \kahler form $\omz$, the \krf \ $\omt$ starting from $\omz$ exists for $t \in [0,\infty)$ if and only if the canonical bundle $K_X$ is nef.
Furthermore, in this setting, the cohomology class $\alpha_t $ of $\omt$ satisfies 
\begin{align*}
	\alpha_t = e^{-t} [\omz] + (1-e^{-t}) 2\pi c_1(K_X) \to 2\pi c_1(K_X) \text{ as }t \to \infty.
\end{align*}
In particular, the \krf \  $\omt$ satisfies (\ref{omi}) if $K_X$ is nef.
\end{theorem}

We now focus on (\ref{Ricci}). 
In general, it is a hard problem to estimate the Ricci curvature along the \krf. 
However, we can reduce (\ref{Ricci}) to the estimate for the scalar curvature.
More precisely, we have the following proposition:
\begin{proposition}\label{ReductionScalar}
If the canonical bundle $K_X$ is nef, then there exists a constant $C>0$ such that for any $T>0$, we have the following estimate:
	\begin{align*}
	\int_0^T dt \int_X \absomt{\ric(\omt) + \omt}^2 \omt^n
		&\le \int_0^T dt\int_XR(\omt )^2 \omt^n +C,
	\end{align*}
Here $C>0$ is a constant which depends only  on the cohomology classes $[\omz]$ and $c_1(K_X)$.
\end{proposition}
\begin{proof}
Recall that the scalar curvature and the volume form evolves as
\begin{align*}
	\dt R(\omt)
		&=\lapomt R(\omt ) + \absomt{\ric(\omt) + \omt}^2 -( R(\omt) + n) , \\
	\dt \omtn 
		&= -(R(\omt) + n) \omtn,
\end{align*}
(for instance \cite[(3.56)]{BEG13IntrotoKRF}). 
Then, we get 
\begin{align*}
&\int_X \absomt{\ric(\omt) + \omt}^2 \omt^n
	=\int_X 
			\left(
				\dt R(\omt)
			\right) \omtn
		+\int_X  ( R(\omt) + n) \omtn \\
	&=\left( 
	\ddt \int_X R(\omt) \omt^n
		+\int_XR(\omt )(R(\omt) + n) \omt^n
	\right)
		-\ddt \int_X \omt^n\\
	&= \ddt 
			\left( -n \left( 2\pi c_1( K_X ) \cdot \alpha_t^{n-1 } \right)  \right)
			+\int_XR(\omt )^2 \omt^n
			- n^2 (\ckx \cdot \alpha_t^{n-1})
			-\ddt \left( \alpha^n_t \right)\\
	&\le \ddt 
			\left( -n \left( 2\pi c_1( K_X ) \cdot \alpha_t^{n-1 } \right)  \right)
			+\int_XR(\omt )^2 \omt^n
			-\ddt \left( \alpha^n_t \right).
\end{align*}
We remark that since $K_X$ is nef and $\alpha_t$ is K\"ahler, $ (\ckx \cdot \alpha_t^{n-1}) \ge 0$ holds for any $t\ge 0$.
Therefore, by integrating with respect to $t$, we obtain the conclusion.
\end{proof}

The following new estimate gives the desired bound for the scalar curvature.
\begin{theorem}\label{Scalar}
Assume that the canonical bundle $K_X$ is semi-positive and not big. 
Let $\omz$ be a \kahler form on $X$ satisfying $[\omz] - 2\pi c_1(K_X) >0$ and $\omt $ be the \krf \ starting from $\omz$.
Let $\Omega $ be a smooth volume form on $X$ such that $-\ricOm \ge 0$.
We set $f_t$ and $\et $ by 
\begin{align*}
	f_t 
		&\deq \log \dfrac{\omtn}{\Ome}, \
	\et 
		\deq \intx \ii \partial \ft \wedge \dbar \ft \wedge \omt^{n-1}. 
\end{align*}
Then the following estimate holds: 
\begin{align*}
	\int_0^\infty  dt \intx \rt^2 \omtn
		&\le \dfrac{n}{2} \ez 
			+C,
\end{align*}
where $C>0$ is a constant depends only on $\omz$ and $c_1(K_X)$.
\end{theorem}
\begin{proof}
We first note that $f_t$ satisfies
\begin{align*}
		\dt \ft 
			&= \tromt{\dt \omt } = -(\rt + n),\ \
		\ddb \ft 
			= -\rict + \ricOme.
\end{align*}
Then, we get the following:
\begin{align}
\notag 
	\ddt \et
		&=-2\intx \fdt \ddb \ft \wedge \omt ^{n-1} 
		\\\notag &\quad 	
		-\intx \ft \ddb \ft \wedge
				(n-1) \omt^{n-2} \wedge \dt \omt \\
\notag 
		&=-2\intx  -(\rt + n) ( -\rict + \ricOme) \wedge \omt ^{n-1} 
		\\\notag &\quad 	
			+(n-1)\intx \ii \partial \ft \wedge \dbar \ft \wedge \omt ^{n-2}\wedge (-\rict -\omt )\\
\label{ABC}
		&= -\dfrac{2}{n}\intx \rt^2 \omtn 
			-2 \intx \rt (-\ricOme )\wedge \omt ^{n-1}
		\\\notag &\quad			
			+(n-1)\intx \ii \partial \ft \wedge \dbar \ft \wedge \omt ^{n-2}\wedge (-\rict -\omt ).	
\end{align}

The second term of (\ref{ABC}) is estimated as follows: Let $C>0$ be a constant satisfying 	$\rt \ge -C$ for $t \in [0,\infty)$ which always exists by a maximum principle argument and only depends on $\omz$ (see \cite[Theorem 3.2.2]{BEG13IntrotoKRF}).
Since the volume form $\Omega $ satisfies $-\ricOme\ge 0$, the second term is estimated as
	\begin{align*}
		\intx \rt (-\ricOme )\wedge \omt ^{n-1}
		&= \intx (\rt +C) (-\ricOme )\wedge \omt ^{n-1} - C \intx (-\ricOme )\wedge \omt ^{n-1} \\
		&= \intx (\rt +C) (-\ricOme )\wedge \omt ^{n-1} -C (\ckx \cdot \alpha_t^{n-1}) \\
		&\ge  -C (\ckx \cdot \alpha_t^{n-1})  \\
		&\ge -C^\prime  e^{-(n-\nu)t},
	\end{align*}
where $\nu$ is the numerical dimension of $K_X$, i.e. $\nu \deq \max \{ k=0, \dots, n \mid (c_1(K_X)^k \cdot [\omz]^{n-k} ) \neq 0 \}$.
Since $K_X$ is not big, we have $\nu < n$.

The third term of (\ref{ABC}) is less than or equal to zero since $[-\rict - \omt] = -\emt ( [\omz ] - \ckx) < 0$. 

Therefore,  we obtain
\begin{align*}
	\ddt \et
		\le -\dfrac{2}{n}\intx \rt^2 \omtn 
			+ C^\prime e^{-(n-\nu)t}.
\end{align*}
By integrating $t$ from $0$ to $\infty$, we get the conclusion since $\et \ge 0$.
\end{proof}

\begin{proof}[Proof of Theorem \ref{Main}]
We assume that $K_X$ is semi-positive and not big.
Let $\omt$ be the \krf \ as in Theorem \ref{Scalar}.
We now prove that $\omt$ satisfies (\ref{omi}), (\ref{Ricci}).
Proposition \ref{Existence} and the semi-positivity of $K_X$ imply (\ref{omi}).
By Proposition \ref{ReductionScalar} and Theorem \ref{Scalar}, we have 
\begin{align*}
	\int_0^\infty dt \int_X \absomt{\ric(\omt) + \omt}^2 \omt^n
		< \infty .
\end{align*}
Then, we can find a sequence $\{ t_i \}_{i=1}^\infty \subset \rr$ such that $t_i \to \infty$ as $i \to \infty$ and $\{ \omti \}_{i=1}^\infty $ satisfies (\ref{Ricci}).
Therefore, by Proposition \ref{Formula}, we obtain (\ref{MY}).
\end{proof}

\bibliographystyle{./00Setting/amsalphaurlmod}
\bibliography{./00Setting/reference}
\end{document}